\documentclass[12pt]{amsart}
\usepackage{amssymb,tikz}
\allowdisplaybreaks

\def\id{\mathop{\rm id}\nolimits}
\def\XX#1{#1}

\def\sat{{\rm sat}}
\def\Lsat{L^\sat}
\newcommand{\set}[2]{\{#1:#2\}}
\def\freeproduct{*}
\def\U{\mathfrak{U}}
\def\sset{\subseteq}
\def\D{\Delta}
\let\iso\cong
\let\cong\equiv

\def\I{{\rm I}}
\def\II{{\rm II}}
\def\kronecker#1#2{(\frac{#1}{#2})}
\def\jacobi#1#2{(\frac{#1}{#2})}
\def\Bigjacobi#1#2{\Bigl(\frac{#1}{#2}\Bigr)}
\def\Orth{{\rm O}}
\def\GL{{\rm GL}}

\def\Z{\mathbb{Z}}
\def\Q{\mathbb{Q}}
\def\N{\mathbb{N}}
\def\F{\mathbb{F}}
\def\R{\mathbb{R}}
\def\tensor{\otimes}
\def\signature{\mathop{\rm signature}\nolimits}
\def\oddity{\mathop{\rm oddity}\nolimits}
\def\std{\mathop{\rm std}\nolimits}
\def\pexcess{\mathop{\hbox{$\pspace$-excess}}\nolimits}
\def\pexcesses{\mathop{\hbox{$\pspace$-excesses}}\nolimits}
\def\pspace{p\kern.1em}
\def\Lneg{L^{\rm neg}}
\def\G{\mathfrak{G}}

\newtheorem{theorem}{Theorem}
\newtheorem{lemma}[theorem]{Lemma}

\theoremstyle{remark}

\begin{document}

\title{Prenilpotent pairs in the $E_{10}$ root lattice}
\author{Daniel Allcock}
\thanks{Supported by NSF grant DMS-1101566}
\address{Department of Mathematics\\University of Texas, Austin}
\email{allcock@math.utexas.edu}
\urladdr{http://www.math.utexas.edu/\textasciitilde allcock}
\subjclass[2000]{%
Primary: 19C99% Steinberg groups and K2, "none of the above but in this section"
; Secondary: 20G44%     Kac-Moody groups
}
\date{February 7, 2017}
%\date{January 21, 2017}
%\date{August 24, 2016}
%\date{September 23, 2014}
%\date{August 15, 2014}

\begin{abstract}
  Tits has defined Kac--Moody groups for all root systems, over all
  commutative rings with unit.  A central concept is the idea of a prenilpotent
  pair of (real) roots.  In particular, writing down his group presentation
explicitly   would require knowing all the Weyl-group orbits of such
pairs.  We show that for the hyperbolic root system $E_{10}$ there are
so many orbits that any attempt at direct enumeration is impractical.
Namely, the number of orbits of prenilpotent pairs having inner
product $k$ grows at least as fast as $(\hbox{constant})\cdot k^7$ as
$k\to\infty$.  Our purpose is to motivate alternate approaches to
Tits' groups, such as the one in \cite{Allcock-amalgams}.
\end{abstract}
\maketitle

\noindent
Kac--Moody groups generalize reductive algebraic groups to include the
infinite dimensional case.  Various authors have defined them in many
ways, the most comprehensive approach being due to Tits \cite{Tits}.
Given a generalized Cartan matrix $A$, he defined a functor
$\tilde{\G}_{\!A}$ assigning a group to each commutative ring~$R$ with
unit.  The main result of \cite{Tits} is that any functor from
commutative rings to groups, having some properties that are
reasonable to expect of anything called a Kac--Moody group, must agree
with $\tilde{\G}_{\!A}$ over every field.  (See \cite[Theorems 1
  and~1$'$]{Tits}.)
%To be able to refer to it, we call
%$\tilde{\G}_{\!A}$ the ``Kac--Moody group associated to $A$'', even
%though it does not satisfy all of his axioms.  

(Actually Tits defined a group functor $\tilde{\G}_D$ for a root
datum~$D$.  For~$\tilde{\G}_{\!A}$ we use the root datum with
generalized Cartan matrix~$A$, which is ``simply connected in the
strong sense'' \cite[p.~551]{Tits}.  The difference between a root datum
and its generalized Cartan matrix plays no role in this paper.)

Tits defined $\tilde{\G}_{\!A}(R)$ by a complicated
implicitly described presentation.  The
key relations are his generalizations of the Chevalley relations.
He begins with the free product
$\freeproduct_{\alpha}(\U_\alpha)$, where $\alpha$ varies over
all real roots and each $\U_\alpha$ is isomorphic to the additive
group of~$R$.
This step requires knowing all the real roots, which is nontrivial but
reasonably accessible (lemma~\ref{lem-what-the-roots-are}).
He imposes relations of the form
$$
[X_\alpha(t),X_\beta(u)]=
\prod_{\gamma}X_\gamma(v_\gamma)
$$ whenever $\alpha,\beta\in\Phi$ form a \emph{prenilpotent pair} (see
below).
Here $\U_\alpha=\set{X_{\alpha}(t)}{t\in R}$ and similarly for the
other roots, the $\gamma$'s parameterizing the product are the real
roots in $\N\alpha+\N\beta$ other than $\alpha$ and $\beta$, and the
parameters $v_\gamma$ depend on various choices like the ordering of
the factors (and anyway are unimportant
in this paper).

The definition of prenilpotency is that some element of the Weyl
group~$W$ sends both $\alpha,\beta$ to positive roots, and some other
element sends both to negative roots.  When this holds, Prop.~1 of
\cite{Tits} and its proof show how to work out the Chevalley relation
of $\alpha,\beta$, at least in principle.  It is similar to, but more complicated
than, the working out of the structure constants of the Kac--Moody
algebra.  (In fact H\'ee \cite{Hee} and Morita \cite{Morita} have worked out all the possible types of the
relations in closed form.)  So the
essence of writing down Tits' presentation is to list all the
prenilpotent pairs.  It would even be enough to find one
representative of each $W\!$-orbit of prenilpotent pairs.  Our main
result,
theorem~\ref{thm-number-of-pairs-at-least-polynomial-of-degree-7}
below, is that this is impossible in practice for the $E_{10}$ root
system, whose Dynkin diagram is
\begin{equation}
  \label{eq-E10-diagram}
  \def\myscale{.6}
\def\mywidth{.8pt}
\def\vertexradius{.1}
\def\vertex(#1){\fill (#1) circle (\vertexradius)}
\begin{tikzpicture}[line width=\mywidth, scale=\myscale]
\draw (0,0) -- (8,0);
\draw (2,0) -- (2,1);
\vertex (0,0);
\vertex (1,0);
\vertex (2,0);
\vertex (3,0);
\vertex (4,0);
\vertex (5,0);
\vertex (6,0);
\vertex (7,0);
\vertex (8,0);
\vertex (2,1);
\end{tikzpicture}%
\end{equation}
The argument suggests
that the same holds for all hyperbolic root systems of rank${}>3$;
see section~\ref{sec-other-lattices}.

This negative result is balanced by the fact that in many interesting
cases, including $E_{10}$, most of the prenilpotent pairs can be
ignored because their Chevalley relations follow from those of other
prenilpotent pairs.  There are two approaches to this.  The first is
due to Abramenko and M\"uhlherr
\cite{Abramenko-Muhlherr}\cite{Caprace}\cite{Muhlherr}, and applies to
Kac--Moody groups associated to $2$-spherical Dynkin diagrams, over
fields, with some exceptions over $\F_2$ and~$\F_3$.  The second
approach is due to the author
\cite{Allcock-amalgams}\cite{Allcock-affine}; see also \cite{Allcock-Carbone}.
It works over general rings, but requires some conditions on the
diagram.  Both approaches apply to all irreducible affine diagrams (of rank${}\geq3$)
and all simply laced diagrams without $A_1$ components, such as $E_{10}$.  In both
cases the result is that $\tilde{\G}_{\!A}(R)$ is the direct limit of
the family of groups $\tilde{\G}_B(R)$, where $B$ varies over the $1$-
and $2$-node subdiagrams of~$A$.  So
one may discard almost all of Tits'
Chevalley-style relations, without changing the resulting group.  In the author's approach, one
even obtains an explicit presentation (often finite) given in terms of
the Dynkin diagram, for example $\tilde{\G}_{E_{10}}(R)$ and
$\tilde{\G}_{E_{10}}(\Z)$ in theorem~\XX{1} and corollary~\XX{2} of
\cite{Allcock-Carbone}.

\medskip
Now we begin the $E_{10}$-specific material.
We write $\Lambda$ for the root lattice, i.e., the integer span of the
simple roots.
The generalized
Cartan matrix $A$ for $E_{10}$ is got from \eqref{eq-E10-diagram} in
the usual way: $A_{ii}=2$ for each node $i$ of the diagram, and
$A_{ij}=-1$ or~$0$ according to
whether distinct nodes $i$ and~$j$ are joined or not.  This matrix is symmetric, so it may
be regarded as an inner product matrix on $\Lambda$.  For
$x\in\Lambda$, the norm of $x$ means $x\cdot x$, usually written
$x^2$.
The Weyl group $W$ acts on $\Lambda$ by isometries.  We will never
refer to imaginary roots, so we follow Tits \cite{Tits} in using ``root''
to mean ``real root'', i.e., ``$W\!$-image of a simple root''.  Now we
can state our main result:

\begin{theorem}
\label{thm-number-of-pairs-at-least-polynomial-of-degree-7}
Let $N(k)$ be the number of $W\!$-orbits of prenilpotent pairs
of roots in the $E_{10}$ root system with inner product~$k$.  Then for some positive
constant $C$, we have $N(k)\geq C k^7$ for all integers~$k$.
\end{theorem}

The constant is made effective in the proof, although we make no
attempt to optimize it.  The theorem says nothing if $k\leq0$, but
this case is uninteresting because there are no prenilpotent pairs
with $k<-1$, by lemma~\ref{lem-what-the-prenilpotent-pairs-are} below.
The proof shows that the problem of enumerating the prenilpotent pairs
contains an infinite sequence of successively more difficult and less
interesting classification problems in the theory of positive-definite
quadratic forms.  For example, the simplest such problem is the classification of positive-definite 
lattices
of dimension~$8$ and determinant $k^2-4$, which becomes difficult and boring for
quite small~$k$.  Hence our description of the direct enumeration of
prenilpotent pairs as ``impractical''.

\section{Proof}
\label{sec-proof}

\noindent
We will prove
theorem~\ref{thm-number-of-pairs-at-least-polynomial-of-degree-7} by
converting it into a lattice-theoretic problem. 
First we need to describe the roots and prenilpotent pairs entirely in
terms of the root lattice $\Lambda$.

\begin{lemma}
\label{lem-what-the-roots-are}
The roots of the $E_{10}$ root system are exactly the norm~$2$ vectors
of $\Lambda$.
\end{lemma}

\begin{proof}
The simple roots have norm~$2$ because the generalized Cartan matrix
has $2$'s along its diagonal.  The other roots are their $W\!$-images
and therefore have norm~$2$ also.  Now suppose a lattice vector $r$
has norm~$2$; we must show it is a root.  The reflection in $r$,
namely $R:x\mapsto x-(x\cdot r)r$, preserves $\Lambda$ because $x\cdot
r\in\Z$ for all lattice vectors~$x$.  Also, $\Lambda$ has signature
$(9,1)$, so the negative-norm vectors in $\Lambda\tensor\R$ fall into
two components.  Since $r^2>0$, $R$ preserves each component.  Vinberg
\cite{Vinberg} showed that $W$ is the full group of lattice isometries
that preserve each component, so $R\in W$.  Since every reflection in
$W$ is conjugate to a simple reflection, $r$ is $W\!$-equivalent to a
simple root.  So $r$ is a root.
\end{proof}

\begin{lemma}[{\cite[Lemma~\XX{6}]{Allcock-Carbone}}]
\label{lem-what-the-prenilpotent-pairs-are}
Two roots in the $E_{10}$ root system form a prenilpotent pair
if and only if their inner product is${}\geq-1$.  \qed
\end{lemma}

At this point the proof of
theorem~\ref{thm-number-of-pairs-at-least-polynomial-of-degree-7}
becomes entirely lattice-theo\-re\-tic, relying on the theory of
integer quadratic forms to study certain sublattices of $\Lambda$.
We
fix $k\geq-1$ and consider prenilpotent pairs with inner product~$k$.
We write $L$ for the integer span of such a prenilpotent pair; its
inner product matrix is
$\bigl(\begin{smallmatrix}2&k\\k&2\end{smallmatrix}\bigr)$.
We will write $\Orth(L)$ and $\Orth(\Lambda)$ for the orthogonal
groups of $L$ and $\Lambda$, and similarly for other lattices.
  The next
  lemma follows immediately from the previous two.

\begin{lemma}
\label{lem-conversion-to-quadratic-forms-problem}
$N(k)$ equals the number of orbits of isometric embeddings $L\to
\Lambda$, under the group $(\Z/2)\times W$, where $\Z/2$ acts on $L$
by swapping its basis vectors and $W$ acts on $\Lambda$ in the obvious
way.  In particular, $N(k)$ is at least as large as the number of
orbits of sublattices of $\Lambda$ that are isometric to $L$,
under the orthogonal group $\Orth(\Lambda)$.  \qed
\end{lemma}

We begin with an overview of a general method called gluing, used for
studying the embeddings of one lattice into another.  When considering
any particular embedding $L\to\Lambda$, we will usually identify $L$
with its image.  In the current situation, one first studies the
possibilities for the saturation $\Lsat:=(L\tensor\Q)\cap \Lambda$.
In the proofs below we will simplify this step away, by restricting to
the case that $L$ is already saturated.  Then, assuming $\det L\neq0$,
one studies the possibilities for $L^\perp$.  In this step we take
advantage of the fact that $\Lambda$ is unimodular: among other
things, it implies that $\Lsat$ and $L^\perp$ have the same
determinant.  This limits $L^\perp$ to finitely many possibilities.
For each candidate $K$ for $L^\perp$, one then considers the possible
ways to glue $K$ to $\Lsat$ in a manner that yields $\Lambda$.  Gluing
means finding a copy of $\Lambda$ between $K\oplus\Lsat$ and
$K^*\oplus(\Lsat)^*$, in which $K$ and $\Lsat$ are saturated.
(Asterisks indicates dual lattices.)  This step boils down to
analyzing the actions of $\Orth(K)$ and $\Orth(L)$ on the discriminant
groups $\D(K)$ and $\D(L)$ of $K$ and $L$, which are finite abelian groups defined below.

Here are the necessary definitions and background.  A lattice $K$
means a free abelian group equipped with a $\Q$-valued symmetric
bilinear pairing.  $K$ is called integral if this pairing is
$\Z$-valued, and $K$ is called even if furthermore all vectors have
even norm. For example, $\Lambda$ is even.  The determinant of the
inner product matrix of $K$, with respect to a basis, is independent
of basis, and is called the determinant $\det K$ of $K$.  We will
encounter only nondegenerate lattices, meaning those of nonzero
determinant.  So we assume nondegeneracy henceforth.  The dual $K^*$
of $K$ means the set of vectors in $K\tensor\Q$ having integer inner
product with all elements of~$K$.  When $K$ is integral, we have
$K\sset K^*$, and in this case we define the discriminant group
$\D(K)$ as $K^*/K$, an abelian group of order $\det K$.  The
$\Z$-valued inner product on $K$ extends to a $\Q$-valued inner
product on $K^*$, which descends to a $\Q/\Z$-valued inner product on
$\D(K)$.  Similarly, if $K$ is even then we can regard the norm of an
element of $\D(K)$ as a well-defined element of~$\Q/2\Z$.  (If $K$ is
not even then the norm is only well-defined mod~$\Z$, but we will only
encounter even lattices.)  By the naturality of the constructions,
$\Orth(K)$ acts on $\D(K)$, preserving these structures.  We write
$\Orth(\D(K))$ for the group of isometries of $\D(K)$, i.e., abelian
group automorphisms that respect the $\Q/\Z$-valued inner product and
$\Q/2\Z$-valued norm.  So there is a natural map
$\Orth(K)\to\Orth(\D(K))$.

The formulation of the theory of integer quadratic forms best suited
for explicit computation is due to Conway and Sloane
\cite[ch.~15]{SPLAG}\cite{Conway-Sloane-mass-formula}.  So we assume
familiarity with their methods and we follow their conventions,
including the unusual one of writing $-1$ for the infinite place of
$\Q$ and defining $\Z_{-1}$ and $\Q_{-1}$ to be $\R$.  For any place
$p$ of $\Q$ we write $K_p$ for the $\pspace$-adic lattice
$K\tensor\Z_p$.  The Sylow $\pspace$-subgroup of $\D(K)$, with its
norm form, is the same as the discriminant group of $K_p$.  Two
lattices $K,K'$ are said to lie in the same genus if $K_p\iso K'_p$
for all places~$p$.  In the positive definite case, the {\it mass} of
a genus means $\sum_K 1/|\Orth(K)|$, where $K$ varies over the
isometry classes in that genus.  This definition makes sense because a
genus contains only finitely many isometry classes.  The mass is
important for us because it is a lower bound for the number of
isometry classes.  We will compute it by using the
Smith--Minkowski--Siegel mass formula, which avoids having to first
enumerate the isometry classes in the genus.

Conway and Sloane gave an elaborate notational system for isometry
classes of $\pspace$-adic lattices, for example the symbols appearing
in lemma~\ref{lem-mass-computations}\ref{item-genus-exists} below.
For $p$ a prime, a symbol $(p^e)^{\pm n}$ indicates an $n$-dimensional
$p$-adic lattice that is got from some unimodular $\pspace$-adic
lattice $U$ by multiplying the inner product by $p^e$.  When $p=2$,
the  symbol also has a subscript, discussed below.  A
chain of symbols $(p^e)^{\pm n}$ represents a direct sum decomposition
of a $p$-adic lattice into Jordan constituents.

For any constituent, the sign~$\pm$, together with the subscript when
$p=2$, describes the isometry class of $U$.  The sign is defined as
the Jacobi symbol $\jacobi{\det U}{p}=\pm$, which we recall is the
Legendre symbol when $p$ is odd.  When $p=2$ it is $+$ or~$-$
according to whether $\det U\cong\pm1$ or $\pm3$ mod~$8$.  In both
cases, the sign is often suppressed when it is~$+$.  When $p=2$ and
$U$ is even, the subscript is the formal symbol~$\II$ and the Jordan
constituent is said to have type~$\II$.  If $p=2$ and $U$ is not even,
then the subscript is an element of $\Z/8$, namely the trace (mod~$8$)
of its
inner product matrix, after diagonalization over $\Z_2$.  (One shows
that a diagonalization does exist, and that the trace is independent of
the choice of diagonalization.)  In this case the Jordan constituent
is said to have type~$\I$.

Now we turn to the specific problem of enumerating the embeddings
$L\to \Lambda$, where we recall that $L$ has inner product matrix
$\bigl(\begin{smallmatrix}2&k\\k&2\end{smallmatrix}\bigr)$.  In fact
  we will bound from below the number of saturated embeddings, meaning
  those whose images are saturated sublattices of $\Lambda$.  We take
  $k\geq3$ to make $L$ indefinite, avoiding the special cases
  $k=-1,0,1,2$.  We factor $d:=-\det L=k^2-4>0$ as
  $2^{e_2}3^{e_3}5^{e_5}\cdots$ and write $f_p$ for the
  non-$\pspace$-part $d/p^{e_p}$ of $d$.  
  
\begin{lemma}
\label{lem-mass-computations}
With $L$, $d$, $e_p$ and $f_p$ as above,
\begin{enumerate}
\item
\label{item-excluded-values-of-e-2}
$e_2$ is $0$, $2$ or${}\geq5$.
\item
\label{item-genus-exists}
There exists a genus of $8$-dimensional positive-definite lattices $K$ of
determinant~$d$, such that
\begin{align*}
K_2
{}&\iso
\begin{cases}
\rlap{\kern110pt if $e_2=0$}
1^{-8}_\II
\\
\rlap{\kern110pt if $e_2=2$}
1^6_\II\, 2^{\kronecker{f_2}{2}2}_{f_2-1}
\\
\rlap{\kern110pt if $e_2\geq5$}
1^6_\II\, 2^1_{-1} \bigl(2^{e_2-1}\bigr)^{\kronecker{f_2}{2}1}_{f_2}
\kern80pt% for centering
\end{cases}
\\
K_{p}
{}&\iso
\begin{cases}
\rlap{\kern110pt for $p>2$ when $e_p=0$}
1^{\kronecker{f_p}{p}8}
\\
\rlap{\kern110pt for $p>2$ when $e_p>0$}
1^{\kronecker{2}{p}7}\,\bigl(p^{e_p}\bigr)^{\kronecker{2f_p}{p}1}
\end{cases}
\end{align*}
\item
\label{item-every-member-of-genus-occurs}
Suppose $K$ lies in this genus.  Then there are at least
$$
\frac{2^{\hbox{\scriptsize\rm\,number of odd primes       dividing $d$}}}{4\,|\Orth(K)|}
$$
$\Orth(\Lambda)$-orbits on the set of saturated sublattices of $\Lambda$ that are isometric
to $L$ and have orthogonal complement isometric to~$K$.
\end{enumerate}
\end{lemma}

\begin{proof}
\ref{item-excluded-values-of-e-2} If $k$ is odd then so is $d=k^2-4$, so $e_2=0$.  If $k$ is divisible by~$4$ then $d$ is divisible
by~$4$ but not~$8$, so $e_2=2$.  If $k$ is twice an odd number then
$d=4({\rm odd}^2-1)$ and the second factor is divisible by~$8$.

As preparation for \ref{item-genus-exists} and
\ref{item-every-member-of-genus-occurs}, we give the $\pspace$-adic invariants of $L$.
Its determinant and signature are $-d$ and $0$, and
\begin{align*}
L_2
{}&\iso
\begin{cases}
\rlap{\kern110pt if $e_2=0$}
1^{-2}_\II
\\
\rlap{\kern110pt if $e_2=2$}
2^{\kronecker{-f_2}{2}2}_{1-f_2}
\\
\rlap{\kern110pt if $e_2\geq5$}
2^1_1 \bigl(2^{e_2-1}\bigr)^{\kronecker{-f_2}{2}1}_{-f_2}
\kern80pt% for centering
\end{cases}
\\
L_{p}
{}&\iso
\begin{cases}
\rlap{\kern110pt if $p>2$ and $e_p=0$}
1^{\kronecker{-f_p}{p}2}
\\
\rlap{\kern110pt if $p>2$ and $e_p>0$}
1^{\kronecker{2}{p}1}\,\bigl(p^{e_p}\bigr)^{\kronecker{-2f_p}{p}1}
\end{cases}
\end{align*}
These can be worked out explicitly using the methods of
\S4.4 and \S7 of \cite[ch.~15]{SPLAG}.
(It helps to observe that if $k$ is even or $p$ is odd then
$L_p\iso\langle2\rangle\oplus\langle-d/2\rangle$.)
Defining $\Lneg$ as $L$ with all inner products negated, its
local forms $\Lneg_{p\neq-1}$ are as follows.  If $e_p=0$ then
$\Lneg_p$ is isometric to $L_p$.
If $e_p>0$ then the Conway--Sloane symbol of $\Lneg_p$ is got from that of $L_p$ by
multiplying each superscript by
$\kronecker{-1}{p}$ (if $p>2$), or negating subscripts (if $p=2$).
When $p=2$ and a constituent has type~$\II$, the negation of
its subscript~$\II$ is taken to mean $\II$ again.

By Theorem~11 in \S7.7 of \cite[ch.~15]{SPLAG}, there
exists a $\Z$-lattice $K$ of determinant~$d$, having specified local
forms $K_{p=-1,2,3,\dots}$, if and only if both the following hold.
First, $\det K_p\in d\cdot(\Q_p^\times)^2$ for all places $p$.
Second, the oddity formula holds:
$$
\signature(K_{-1})+\sum_{p\geq3}\pexcess(K_p)
\cong
\oddity(K_2)
\quad
\pmod8.
$$ Here the oddity (resp.\ $\pexcess$) is a $\Z/8$-valued invariant of
quadratic forms over $\Q_2$ (resp.\ $\Q_{p\,\rm odd}$), defined
in \S5.1 of \cite[Ch.~15]{SPLAG}.  It is the
sum of the oddities (resp.\ $\pexcesses$) of the Jordan constituents,
which can be read off from the Conway-Sloane notation as follows.
The oddity of a $2$-adic Jordan constituent of type~$\II$ is
always~$0$.  For a type~$\I$ constituent $(2^e)^{\pm n}_t$, the oddity
is the subscript $t$, plus~$4$ if
the sign is~$-$ and $e$ is odd.
For odd $p$, the $\pexcess$ of $(p^e)^{\pm n}$ is
$n(p^e-1)$, plus~$4$ if the sign is~$-$ and $e$ is odd.

Both the determinant condition and
the oddity formula for the family of $K_p$'s in \ref{item-genus-exists}
can be verified as follows.  First,
  $K_{-1}$ and $L_{-1}$ have signature~$8$ and determinants of
opposite signs.
Second, although we didn't say
so, we constructed
$K_2$ as $1^6_\II\oplus\Lneg_2$ and
$K_p$ as $1^{\kronecker{-1}{p}6}\oplus\Lneg_p$ for $p>2$.  
Now, the $\Z_2$-lattice $1^6_\II$ is isometric to the sum of three copies of
$\bigl(\begin{smallmatrix}0&1\\1&0\end{smallmatrix}\bigr)$, so it has
  determinant~$-1$.  And for odd~$p$, the $\Z_p$-lattice
  $1^{\kronecker{-1}{p}6}$ is isometric to $\langle1,1,1,1,1,-1\rangle$, so it also
  has determinant~$-1$.
It follows that $\det K_p=-\det\Lneg_p=d$ for all $p$, verifying the
determinant condition.  For the oddity formula, note that
the
$2$-adic lattice $1^6_\II$ has oddity~$0$, and for $p>2$ the
$\pspace$-adic lattice $1^{\kronecker{-1}{p}6}$ has $\pexcess$ equal
to~$0$.  Since $\Lneg$ exists, its local forms $\Lneg_p$ satisfy the
oddity formula.  Since the corresponding formula for the $K_p$ has
exactly the
same terms, it also holds.  So there exists a lattice~$K$ having those
local forms.  

\ref{item-every-member-of-genus-occurs} In the language of \S3 of
\cite[ch.~4]{SPLAG}, this is the question of how one may glue $L$
to~$K$ to obtain~$\Lambda$.  Here are the details.  Suppose that in
addition to $K$, we are given a totally isotropic subgroup $G$ of
$\D(L\oplus K)=\D(L)\oplus\D(K)$ that projects isomorphically onto
$\D(L)$ and onto $\D(K)$.  Totally
isotropic means that the natural $\Q/\Z$-valued inner
product and $\Q/2\Z$-valued norm on $\D(L\oplus K)$ vanish identically
on $G$.
From $K$ and $G$ we will construct a
saturated embedding $L\to\Lambda$ with $L^\perp\iso K$.  

Before constructing the embedding, we explain why such a subgroup $G$
exists. Recall from the proof of \ref{item-genus-exists} that each $K_p$ was
constructed as the sum of $\Lneg_p$ and a unimodular $\Z_p$-lattice.
It follows that $\D(K_p)$ and $\D(\Lneg_p)$ are isomorphic as finite
abelian groups equipped with  $\Q_p/\Z_p$-valued inner products
and $\Q_p/2\Z_p$-valued norms.  These discriminant groups are the
Sylow $p$-subgroups of $\D(K)$ and $\D(\Lneg)$, so it follows that $\D(K)$
and $\D(\Lneg)$ are isomorphic in the corresponding sense: there
exists a group isomorphism $\D(L)\to\D(K)$ that negates norms and
inner products.  We may take $G$ to be its graph.

Here is the construction of the embedding $L\to\Lambda$.  Write
$L\oplus_G K$ for the preimage of $G\sset\D(L)\oplus\D(K)$ in
$L^*\oplus K^*$.  This construction is called ``gluing $L$ to $K$ by
$G$''.  The resulting lattice is integral and even (since $G$ is
totally isotropic) and unimodular (since its index~$d$ sublattice
$L\oplus K$ has determinant~$-d^2$).  By Theorem~5 of \cite[\S
  V.2]{Serre}, up to isometry $\Lambda$  is the only even unimodular lattice of signature
$(9,1)$.  So $L\oplus_G K\iso\Lambda$
and we have constructed a copy of $\Lambda$ containing $L\oplus K$.
Since $G\sset\D(L)\oplus\D(K)$ meets $\D(L)$ and $\D(K)$ trivially, both $L$
and $K$ are saturated in this copy of $\Lambda$.  This completes the
construction of a saturated embedding $L\to\Lambda$ whose orthogonal
complement is isometric to~$K$.

In fact we get such an embedding for every totally isotropic subgroup
$G$ of $\D(L)\oplus\D(K)$ that projects isomorphically to each
summand.  We can usually obtain many such subgroups from the one
constructed above, by taking its images under transformations $f\oplus
\id_{\D(K)}$ where $f$ is an isometry of $\D(L)$.  Distinct $f$'s give
distinct subgroups, because $G\to\D(K)$ is an isomorphism.  The number
of self-isometries of $\D(L)$ is at least~$2^o$, where $o$ is the
number of odd primes $p$ dividing~$d$.  To see this, note that for
each such $p$, the self-map of $\D(L)$ which negates the Sylow
$p$-subgroup, and acts by the identity on all other Sylow subgroups,
is an isometry.  So $L^*\oplus K^*$ contains at least $2^o$ copies of
$\Lambda$ in which $L$ and $K$ are saturated.

It is possible for two of these subgroups $G$, $G'$ to yield equivalent embeddings
$L\to\Lambda$.  ``Equivalent'' has the obvious meaning: that there is
an isometry from $L\oplus_G K\iso\Lambda$ to $L\oplus_{G'}K\iso\Lambda$ that sends $L$ to $L$.  It is easy to see that this happens if and only if there
are isometries of $L$ and~$K$, such that the induced isometry
of $\D(L)\oplus\D(K)$ sends $G$ to $G'$.  To prove \ref{item-every-member-of-genus-occurs}
it therefore suffices to show that the $2^o$ many subgroups of
$\D(L)\oplus\D(K)$ constructed in the previous paragraph represent at
least $2^o/4|\Orth(K)|$ many orbits under the action of
$\Orth(L)\times\Orth(K)$ on $\D(L)\oplus\D(K)$.  And to prove this it
suffices to prove that the action of $\Orth(L)$ on
$\D(L)$ factors through a group of order${}\leq4$.

For this we write down generators for $\Orth(L)$. Consider its
action on the set of norm~$-1$ vectors in $L\tensor\R$.  There are two
components, exchanged by negation, each a copy of $1$-dimensional
hyperbolic space (a copy of the real line).  The subgroup of $\Orth(L)$
generated by the reflections in norm~$2$ roots is normal, and acts on
each component as an infinite dihedral group~$D_\infty$.  Since
$\Orth(L)$ normalizes $D_\infty$, it is generated by $D_\infty$, the
$\Orth(L)$-stabilizer of a Weyl chamber, and negation.  Since the Weyl
chamber is an interval, its $\Orth(L)$-stabilizer has order${}\leq2$,
so $\Orth(L)$ contains $D_\infty$ of index${}\leq4$.  To finish the
proof we check that $D_\infty$ acts trivially on $\D(L)$.
Suppose $x\in L^*$ and that $r\in L$ has norm~$2$.  Then $r$'s
reflection sends $x$ to $x-(x\cdot r)r$, which differs from $x$ by an
element of $L$, hence represents the same element of $\D(L)$ as~$x$.
\end{proof}

\begin{lemma}
\label{lem-mass}
The genus in lemma~\ref{lem-mass-computations} has mass equal to
$$
\frac{d^{\,7/2}\,\zeta_d(4)}{30240\pi^4\cdot2^{\hbox{\scriptsize\rm\,number of odd primes
      dividing $d$}}}
\cdot
\begin{cases}
\frac{1}{272}&\hbox{if $e_2=0$}
\\
\frac{1}{512}&\hbox{if $e_2=2$}
\\
\frac{1}{1024}&\hbox{if $e_2\geq5$}
\end{cases}
$$
where $\zeta_d$ is defined as in
\cite[\S7]{Conway-Sloane-mass-formula}, by
$$
\zeta_d(s)
=
\prod_{{\rm primes}\,p\,\nmid\,2d}\Bigl\{1-\Bigjacobi{d}{p}\frac{1}{p^s}\Bigr\}^{-1}
=
\sum_{\begin{smallmatrix}m\geq1\\(m,2d)=1\end{smallmatrix}}\Bigjacobi{d}{m}m^{-s}.
$$ 
\end{lemma}

\begin{proof}
  This is a lengthy exercise using the intricate but explicit
  procedure in \cite{Conway-Sloane-mass-formula}.  We will use the
  specialized vocabulary developed there, giving enough
  details that the reader possessing a copy of \cite{Conway-Sloane-mass-formula}, but unfamiliar
  with it, can follow along.  
  Following \cite[\S7]{Conway-Sloane-mass-formula}, the mass is
\begin{equation}
\label{eq-mass-as-standard-mass-times-factors}
m(K)=\std(K)\prod_{{\rm primes\,}p|2d}\frac{m_p(K)}{\std_p(K)}
\end{equation}
	where $m_p(K)$ is the ``$\pspace$-mass'', and $\std(K)$ resp.\ $\std_p(K)$ is the
``standard mass'' resp.\ ``standard $p$-mass'' of an $8$-dimensional lattice of determinant~$d$.  Table~3 of \cite{Conway-Sloane-mass-formula} gives
$\std(K)$ as $\zeta_d(4)/30240\pi^4$, and \S7 of \cite{Conway-Sloane-mass-formula} gives
$$
\std_p(K):=\frac{1}{2(1-p^{-2})(1-p^{-4})(1-p^{-6})}.
$$ For computing $m_p(K)$ we refer to \S4--\S5 of
\cite{Conway-Sloane-mass-formula}.  It is defined as
$$
m_p(K)=(\hbox{diagonal product})\cdot(\hbox{cross
  product})\cdot(\hbox{type factor}),
$$ the last term appearing only if $p=2$.  The diagonal product is the
product of the ``diagonal factors'' $M_p(f)$, where $f$ varies over
the $p$-adic Jordan constituents.  $M_p(f)$ is defined in
\cite[table~2]{Conway-Sloane-mass-formula} in terms of $p$ and the ``species'' of~$f$, which is one of
the formal symbols
\begin{center}
$0+$\quad or\quad $1$, $3$, $5$, $7,\dots$\quad or\quad $2+$, $2-$, $4+$, $4-$, $6+$, $6-,\dots$
\end{center}
The species can be read from the
Conway-Sloane symbol for~$f$, except when $p=2$ when it is also a
function of whether $f$ is ``bound'' or ``free'', which depends on
$f$'s neighboring Jordan constituents.
See table~1 in \cite{Conway-Sloane-mass-formula}.
(When $p=2$, one sometimes also refers to the ``octane value'' of~$f$ when
determining the species.  This is an element of $\Z/8$ which can be
read from the Conway-Sloane symbol for~$f$.  We will explain how, when
we need to.)
Even the $0$-dimensional Jordan constituents contribute to the
diagonal product, but their diagonal factors are~$1$ except when they
are bound and $p=2$.  These exceptional constituents are called
``bound love forms'', which have diagonal factor~$\frac12$.

The cross product is the product of the ``cross terms'', one for each
pair of distinct $p$-adic Jordan constituents.  Given where $e<e'$,
the cross-term for $(p^e)^{\pm n}$ and $(p^{e'})^{\pm n'}$ is $p^{n
  n'(e'-e)/2}$.  The type factor appears only when $p=2$, when it is
$2^{n(\I,\I)-n(\II)}$.  Here $n(\I,\I)$ means the number of pairs of
adjacent type~$\I$ constituents, and $n(\II)$ means the sum of the
ranks of the type~$\II$ constituents.

Now we begin the computations proper.  We
treat the case of odd $p|d$ first.  The 
Jordan constituents
$1^{\kronecker{2}{p}7}$ and
$\bigl(p^{e_p}\bigr)^{\kronecker{2f_p}{p}1}$ have species $7$ and $1$
respectively.  So their diagonal factors are
$M_p(7)=\std_p(K)$ and $M_p(1)=\frac12$.
Since there are only two Jordan constituents, there is a single 
cross-term, namely
$p^{7\cdot1\cdot(e_p-0)/2}=p^{7e_p/2}$.  By
definition, $m_p(K)$ is the product of the diagonal factors and this
cross-term.  This gives $m_p(K)/\std_p(K)=\frac{1}{2}p^{7e_p/2}$.

The calculation of $m_2(K)$ is similar but more intricate, and we must
treat all three possibilities for $K_2$ listed in lemma~\ref{lem-mass-computations}\ref{item-genus-exists}.  First suppose $e_2=0$, so
$K_2\iso 1^{-8}_\II$.  The single Jordan constituent is free with
type~$\II$, dimension~$8$ and sign~$-$.
The octane value of a type~$\II$ constituent is $0$ or~$4$ according
to whether the sign is $+$ or~$-$.  Therefore the octane value of
$1_\II^{-8}$ is~$4$, so table~1 of \cite{Conway-Sloane-mass-formula} gives its species as~$8-$,
and then formula (5) of \cite{Conway-Sloane-mass-formula} gives its diagonal factor as
$$
M_2(8-)
=
\frac{1}{2(1-2^{-2})(1-2^{-4})(1-2^{-6})(1+2^{-4})}
=\frac{16}{17}\std_2(K).
$$ There are no type~$\I$ constituents (hence no bound love forms) and
no cross-terms.  Since type~$\II$ constituents account for~$8$
dimensions, the type factor is~$2^{-8}$.  So
$m_2(K)/\std_2(K)=\frac{16}{17}2^{-8}=\frac{1}{272}2^{7e_2/2}$.

Now suppose $e_2=2$, so $K_2\iso1^6_\II\,
2^{\kronecker{f_2}{2}2}_{f_2-1}$.  The first constituent is bound and
$6$-dimensional of type~$\II$. So it has species~$7$, hence
diagonal factor $M_2(7)=\std_2(K)$.  Before analyzing the second
constituent we remark that $f_2\cong3$ mod~$4$.  To see this, recall 
from the proof of lemma~\ref{lem-mass-computations}\ref{item-excluded-values-of-e-2} that $e_2=2$ exactly when $k=2l$ with $l$ even.  From $d=4(l^2-1)$
we get $f_2=(l+1)(l-1)$, and observe that one factor on the right
is $1$ mod~$4$ while the other is $3$ mod~$4$.  Now, the octane value
of a type~$\I$ constituent is its subscript, plus~$4$ if the
constituent's sign is~$-$.
We have just shown that $f_2\cong3$ or $7$ mod~$8$.  
In either
case,   $2^{\kronecker{f_2}{2}2}_{f_2-1}$ has octane value~$6$, hence species~$1$, hence diagonal
factor $M_2(1)=\frac12$.  There is one bound love form, namely
$4^{+0}_\II$, and its diagonal factor is~$\frac12$.  There is one cross-term, contributing a
factor $2^{6\cdot2/2}$.  There are no pairs of adjacent type~$\I$
constituents, and $6$ dimensions total of type~$\II$ constituents,
so the type factor is~$2^{-6}$.  Therefore
$$
m_2(K)/\std_2(K)=\textstyle
\frac{1}{2}\cdot\frac{1}{2}\cdot2^6\cdot2^{-6}=\frac{1}{512}2^{7e_2/2}.
$$

Finally, suppose $e_2\geq5$, so $K_2\iso
1^6_\II\, 2^1_{-1}
\bigl(2^{e_2-1}\bigr)^{\kronecker{f_2}{2}1}_{f_2}$.  The constituent
$1^6_\II$ is bound, hence has species~$7$, hence diagonal factor
$M_2(7)=\std_2(K)$.  The constituent $2^1_{-1}$ is free with octane
value~$-1$, hence species $0+$, hence diagonal factor $M_2(0+)=1$.
The last constituent $\bigl(2^{e_2-1}\bigr)^{\kronecker{f_2}{2}1}_{f_2}$ is
free.  Considering the four possibilities for $f_2$ mod~$8$
shows that the octane value is always $\pm1$, so this constituent also
has species~$0+$ and diagonal factor~$1$.  There are three bound love
forms, namely $4^{+0}_\II$, $(2^{e_2-2})^{+0}_\II$ and
$(2^{e_2})^{+0}_\II$, with diagonal factors~$\frac12$.
So the diagonal product is $\std_2(K)\cdot1\cdot1\cdot\frac12\cdot\frac12\cdot\frac12$.
There is a  cross-term for
each pair of constituents, and the cross-product is their product,
namely
$$
\textstyle
2^{6\cdot1/2}
\cdot
\bigl(2^{e_2-1}\bigr)^{6\cdot1/2}
\cdot
\bigl(2^{e_2-2}\bigr)^{1\cdot1/2}
=\frac{1}{2}2^{7e_2/2}
$$
Finally, there are no pairs of adjacent type~$\I$ constituents, and $6$
dimensions total of type~$\II$ constituents, so the type
factor is~$2^{-6}$.  Multiplying the diagonal product, cross product
and type factor together yields
$$
\textstyle
m_2(K)/\std_2(K)=2^{-3}\cdot\frac{1}{2}2^{7e_2/2}\cdot2^{-6}
=\frac{1}{1024}2^{7e_2/2}
$$

We have now computed all the ingredients in 
\eqref{eq-mass-as-standard-mass-times-factors}, and assembling them
yields the lemma.
\end{proof}

\begin{proof}[Proof of theorem~\ref{thm-number-of-pairs-at-least-polynomial-of-degree-7}]
  By lemma~\ref{lem-conversion-to-quadratic-forms-problem}, $N(k)$ is at least as large as the number of
  $\Orth(\Lambda)$-orbits on saturated sublattices of $\Lambda$ that
  are isometric to $L$.  By lemma~\ref{lem-mass-computations}\ref{item-every-member-of-genus-occurs}, the latter quantity
  is at least
  $$
\sum_K  \frac{2^{\hbox{\scriptsize\rm\,number of odd primes dividing $d$}}}{4\,|\Orth(K)|}
  $$
  where $K$ varies over the genus studied there.  The
  number of terms in the sum is the size of the genus. This is at
  least twice the mass given in lemma~\ref{lem-mass}, because each lattice
  has at least two isometries.  So the number of terms is at least
  $$
  \frac{2\,d^{7/2}\,\zeta_d(4)}{1024\cdot30240\pi^4\cdot
    2^{\hbox{\scriptsize\rm\,number of odd primes dividing $d$}}}
  $$
  Also, replacing $\Orth(K)$ in every term by
  $W(E_8)$ does not increase the sum,
  because the largest possible order for a finite subgroup of
  $\GL_8(\Z)$ is $|W(E_8)|$ (see
  \cite{Conway-Sloane-subgroups-of-GL-n-Z} and its references).
  Therefore
  $$
N(k)\geq  \frac{2 d^{7/2}\,\zeta_d(4)}{1024\cdot30240\pi^4\cdot4|W(E_8)|}
  $$
  Next we note
  $$
  \zeta_d(4)\geq1-\frac{1}{2^4}-\frac{1}{3^4}-\cdots
  =2-\sum_{n=1}^\infty n^{-4}=2-\pi^4/90
  $$
We have shown that
$$
N(k)\geq
\frac{2(k^2-4)^{7/2}\bigl(2-\pi^4/90)}{1024\cdot30240\pi^4\cdot4|W(E_8)|}
> 2.1\times10^{-19}(k^2-4)^{7/2}
$$ whenever $k\geq3$.  This almost proves our claim that the function $N(k)$ is
bounded below by $Ck^7$ for some constant $C>0$.  What remains is to
check that $N(1)$ and $N(2)$ are positive.  The $E_{10}$ root system
contains an $A_2$ (resp.\ $E_9$) root system, so it contains a pair of roots with
inner product~$1$ (resp.\ $2$).  Therefore $N(1)$ and~$N(2)$ are
positive, as desired.
\end{proof}

\section{Other hyperbolic root lattices}
\label{sec-other-lattices}

\noindent
The details of the previous section were $E_{10}$-specific, but the
same philosophy looks likely to apply to the other symmetrizable hyperbolic root
systems.  This suggests the same enumeration-is-impracticable conclusion in
rank${}>3$.  We have not worked out the details, because for us the
$E_{10}$ result is enough to motivate the improvements to Tits'
presentation that we mentioned in the introduction.  But it seems
valuable to give an outline of how the calculations would go.

By a hyperbolic root system we mean one arising from an irreducible
Dynkin diagram that is neither affine nor spherical, but whose irreducible proper subdiagrams are.  There are 238 such Dynkin
diagrams, of which 142 are symmetrizable; see \cite{Carbone-etal}.
Symmetrizability is equivalent to the root lattice $\Lambda$ possessing an
inner product that is invariant under the Weyl group~$W$.  This is obviously a
prerequisite to applying lattice-theoretic methods.  Hyperbolicity
implies that $\Lambda$ has Lorentzian signature and that $W$ has finite
index in $\Orth(\Lambda)$.

The roots are the $W\!$-images of the simple roots, so there are only
finitely many root norms.  For each pair of such norms $N,N'$, we can study
prenilpotent pairs of roots $r,r'$ with norms $N,N'$.  The analogue of
lemma~\ref{lem-what-the-prenilpotent-pairs-are} is that $r,r'$ form a prenilpotent pair if and only if
$k:=r\cdot r'$ is larger than $-\sqrt{N N'}$.  By taking $k>\sqrt{N N'}$ we may suppose the span
$L$ of $r,r'$ is indefinite.  We are interested in the number $N(k)$
of $W\!$-orbits of such prenilpotent pairs.

Next one studies the embeddings of $L$ into $\Lambda$ as in
lemma~\ref{lem-mass-computations}, which of course depend on $d:=-\det L\approx k^2$.  One
can follow the $E_{10}$ argument to bound below the number of
$\Orth(L)$-orbits of saturated copies of $L$ in~$\Lambda$.  First one would
have to work out which genera could occur as $L^\perp$.  If there are
any, then we fix one and and restrict attention to saturated copies of
$L$ for which $L^\perp$ lies in that genus.  Then one would work out
the mass of that genus.  The essential part of the mass calculations
in lemma~\ref{lem-mass} are the cross-terms, because they provide the
$d^{7/2}$ term that yields theorem~\ref{thm-number-of-pairs-at-least-polynomial-of-degree-7}.  The corresponding term for $\Lambda$
would be $d^{(\dim\Lambda-3)/2}$.  This suggests that the number of
$\Orth(L)$-orbits of prenilpotent pairs (with $N$, $N'$ fixed as
above) grows at least as fast as a multiple of $k^{\dim\Lambda-3}$.

An obstruction to turning this into a proof is that there may be some
embeddings of $L$ into $\Lambda$ that send the basis vectors to
non-roots.  We expect that the finiteness of
$[\Orth(\Lambda):W]$ means that this difficulty can  be more or less ignored.
The point is that each $\Orth(\Lambda)$-orbit of embeddings $L\to\Lambda$
splits into at most $[\Orth(\Lambda):W]$ many $W\!$-orbits.
So we expect that there is a positive constant~$C$, such that
for each $k$,  $N(k)$ is either~$0$ or at least
$C k^{\dim\Lambda-3}$.

This suggests that if $\dim\Lambda>3$ then tabulating the prenilpotent
pairs is not feasible.  But the $\dim\Lambda=3$ case is borderline and
may be amenable to direct attack.  Indeed, Carbone and Murray
\cite{Carbone-Murray} have studied one particular case with
$\dim\Lambda=3$.  From our perspective, what is special about the
$\dim\Lambda=3$ case is that $L^\perp$ is $1$-dimensional, and every
$1$-dimensional genus has a unique member and mass $1/2$.  So the main
contribution to the analogue of
theorem~\ref{thm-number-of-pairs-at-least-polynomial-of-degree-7}'s
$N(k)$ will be some analogue of the term
\begin{equation}
\label{eq-two-to-power}
  2^{\hbox{\scriptsize\rm\,number of odd primes dividing $d$}}
\end{equation}
from lemma~\ref{lem-mass-computations}\ref{item-every-member-of-genus-occurs}.
As a function of $k$ (recall $d=k^2-4$), this behaves irregularly.
For example, if there are infinitely many primes at distance~$4$ from
each other, then \eqref{eq-two-to-power} takes the value~$4$ infinitely often, even
though it also takes arbitrarily large values.

\end{document}